\documentclass[11pt, a4paper]{article}

\usepackage[T1]{fontenc}
\usepackage{amsmath}
\usepackage{amsthm}
\usepackage{amsfonts}
\usepackage{bbm}

\usepackage{cite}

\usepackage{hyperref}

\theoremstyle{plain}
\newtheorem{thm}{Theorem}

\newtheorem{cor}[thm]{Corollary}
\newtheorem{prop}[thm]{Proposition}
\theoremstyle{definition}
\newtheorem{rem}[thm]{Remark}

\def\P{\mathbb{P}}
\def\E{\mathbb{E}}
\def\R{\mathbb{R}}
\def\eqd{\overset{d}{=}}

\newcommand{\PP}[1]{\P\left({#1}\right)}
\newcommand{\EE}[1]{\E\left[{#1}\right]}

\newcommand{\di}{\,\textrm{d}}

\DeclareMathOperator{\spn}{span}

\def\Z{\mathbb{Z}}

\let\BFseries\bfseries\def\bfseries{\BFseries\mathversion{bold}} %formulas in headings bold

\title{Persistence probabilities of two-sided (integrated) sums of correlated 
stationary Gaussian sequences}

\author{Frank Aurzada \and Micha Buck}

\begin{document} 
\maketitle

\begin{abstract}
We study the persistence probability for some two-sided discrete-time Gaussian sequences that are discrete-time analogs of fractional Brownian motion and integrated fractional Brownian motion, respectively.
Our results extend the corresponding ones in continuous-time in \cite{Molchan1999a} and \cite{Molchan2017a} to a wide class of discrete-time processes.
\end{abstract}

\section{Introduction}

Persistence concerns the probability that a stochastic process has a long negative excursion. In this paper, we are concerned mainly with two-sided discrete-time processes: If $Z = (Z_n)_{n \in \Z}$ is a stochastic process, we study the rate of decay of the probability
\[
\PP{Z_n \leq 0 \ :\ \vert n \vert  \leq N}, \quad \text{as} \quad N \to \infty.
\]
In many cases of interest, the above probability decreases polynomially, i.e., as $N^{-\theta+o(1)}$, and it is the first goal to find the persistence exponent $\theta$. 
For a recent overview on this subject, 
we refer to the surveys \cite{Majumdar1999}, \cite{Bray2013a}, \cite{Aurzada2015a}.

The purpose of this paper is to analyse the persistence probability for the discrete-time analogs of two-sided fractional Brownian motion (FBM) and two-sided integrated fractional Brownian motion (IFBM). 
Our study extends results in \cite{Molchan1999a} and \cite{Molchan2017a}, respectively, to a wide class of discrete-time processes.

The study of persistence probabilities of FBM, IFBM and related processes has received considerable attention in theoretical physics and mathematics, recently. For instance, see \cite{Molchan2004} and \cite{Molchan2017a} where a relation between the Hausdorff dimension of Lagrangian regular points for the inviscid Burgers equation with FBM initial velocity and the persistence probabilities of IFBM is established; the interest for it arises from \cite{She1992} and \cite{Sinai1992}. 
Further, in \cite{Oshanin2013} a physical model involving FBM is studied as an extension to the Sinai model; see also \cite{Aurzada2013b}. Here, persistence probabilities are related to scaling properties of a quantity, called steady-state current.  Moreover, persistence of non-Markovian processes that are similar to FBM are studied in \cite{Castell2013} and \cite{Aurzada2016aUnp}, confirming results in \cite{Redner1997} and \cite{Majumdar2003}.
 
Let us recall that a FBM $(W_H(t))_{t \in \R}$ is a centered Gaussian process with covariance
\[
\EE{W_H(t) W_H(s)} = \frac{1}{2} \left( \vert t \vert^{2H} + \vert s \vert^{2H} - \vert t-s \vert^{2H} \right), \quad t,s\in\R,
\]
where $0 < H < 1$ is a constant parameter, called Hurst parameter. For
$H = 1/2$ this is a usual two-sided Brownian motion. For any $0 < H < 1$, the process has stationary increments, but no independent increments (unless $H = 1/2$). Furthermore, it is an $H$-self-similar process.
An IFBM $(I_H(t))_{t \in \R}$ is defined by $I_H(t) := \int_0^t W_H(s) \di s$ 
and is an $(H+1)$-self-similar process.

In order to define the discrete-time analogs, let $(\xi_n)$ be a real valued stationary centered Gaussian sequence such that
\begin{equation}
\label{eq:varS}
\sum_{j=1}^n \sum_{k=1}^n \E \xi_j \xi_k \sim n^{2H} \ell(n), \quad n \to \infty,
\end{equation}
with $0<H<1$ and $\ell$ slowly varying at infinity. 
Here and below, we write $f(x) \sim g(x)$ $(x \to x_0)$ if $\lim f(x)/g(x)=1$ as $x \to x_0$.
Then,
(\ref{eq:varS}) implies the weak convergence result
\begin{equation}
\label{scalingLimit}
\left( \frac{1}{n^H \ell(n)^{1/2}} \sum_{k=1}^{\lfloor nt \rfloor} \xi_k \right)_{t \geq 0} \Rightarrow (W_H(t))_{t \geq 0}
\end{equation}
with fractional Brownian motion $(W_H(t))$, see e.g.\ Theorem 4.6.1 in \cite{Whitt2002}. 
For this reason, it is natural to consider the stationary increments sequence $(S_n)_{n\in\Z}$ given by \[S_n-S_{n-1}:=\xi_n \text{ for }  n \in \Z \quad \text{and} \quad S_0:=0\] as a discrete-time analog of FBM.

Now, we will define the discrete-time analog of IFBM such that symmetry properties like in the  continuous-time setting are satisfied.
With this in mind, a natural process is given by \[{I_n-I_{n-1}:=(S_n+S_{n-1})/2} \text{ for } n \in \Z \quad \text{and} \quad I_0:=0.\] In Section 2, we discuss relations to the process with increments $(S_n)$ (instead of $((S_n+S_{n-1})/2)$), which may also seem natural but for which our method of proof does not apply directly due to a lack of symmetry. 

In \cite{Molchan1999a} it is shown that one has $\PP{W_H(t) \leq 1 \ :\ \vert t \vert \leq T} = T^{-1 + o(1)}$. 
Our first result, treats the discrete-time analog. The technique we use to prove the theorem is completely different from the one in \cite{Molchan1999a}.

\begin{thm}
\label{thm:S_n}
Let $(\xi_n)$ be a real valued stationary centered Gaussian sequence such that (\ref{eq:varS}) holds. Then, there is a constant $c>0$ such that, for every $N\geq 1$,
\begin{equation*}
c^{-1} N^{-1} \leq \PP{S_n \leq 0 \ :\ \vert n \vert \leq N} \leq N^{-1}.
\end{equation*}
\end{thm}

In order to prove the corresponding result for the process $(I_n)$, we will use a change of measure argument. 
This argument requires an additional assumption as follows:
Let $\mu$ denote the spectral measure of the sequence $(\xi_n)$, i.e.,
\[
\E \xi_j \xi_k =: \int_{(-\pi,\pi]} e^{i \vert j-k \vert u} \di \mu(u).
\]
The spectral measure $\mu$ has a (possibly vanishing) component that is absolutely
continuous with respect to the Lebesgue measure. Let us denote by $p$ its
density, i.e., $\di\mu(u) =: p(u)\di u + \di\mu_s(u)$.
We will assume that $p$ satisfies
\begin{equation}
\label{eq:specDens}
p(u) \sim \ell(1/u) \vert u \vert^{1-2H}, \quad u \to 0,
\end{equation}
where $\ell$ is a slowly varying function at infinity. It is well-known that (\ref{eq:specDens}) implies (\ref{eq:varS}) and thus (\ref{scalingLimit}).

The nature of this assumption can be understood by considering the fractional Gaussian noise process, defined by $\xi_n^{\textsc{fgn}} := W_H(n) - W_H(n - 1)$. This stationary centered Gaussian sequence has an absolutely continuous spectral measure with density function $p_{\textsc{fgn}}$ that satisfies (see e.g. \cite{Samorodnitsky_2006})
\[p_{\textsc{fgn}}(u) \sim m_H \vert u \vert^{1-2H},\quad u \to 0,\]
where $m_H = \Gamma(2H + 1) \sin(\pi H)/2 \pi$.
So, we assume that the density of the absolutely continuous part of the spectral measure of the stationary process $(\xi_n)$ is comparable to the spectral density of fractional Gaussian noise, up to the slowly varying function $\ell$. 

We are now ready to state our second main result.

\begin{thm}
\label{thm:I_n}
Let $(\xi_n)$ be a real valued stationary centered Gaussian sequence such that (\ref{eq:specDens}) holds. Then,
\begin{equation*}
\PP{I_n \leq 0 \ :\ \vert n \vert \leq N} = N^{-(1-H)+o(1)}.
\end{equation*}
\end{thm}

We recall that \cite{Molchan2017a} considers the continuous-time case. Many arguments from that paper can be adapted to our setup. However, for instance, arguments using self-similarity need to be replaced by new ideas. Furthermore, new results concerning the change of measure are needed and may be of independent interest.

For example,
as a byproduct of the change of measure techniques,
we can improve Theorem 11 in \cite{Aurzada2016aUnp}, where the persistence problem of the one-sided discrete-time analog of FBM is considered. 
There it is shown that for every real valued stationary centered Gaussian sequence $(\xi_n)_{n \in \mathbb{N}}$  such that \eqref{eq:varS} holds and every $a > 0$, there is some constant $c>0$ such that
\begin{align}
\begin{aligned}
\label{eq:thm11aurzada}
c^{-1}N^{-(1-H)} &\frac{\sqrt{\ell(N)}}{\sqrt{\log(N)}} \leq \PP{S_n < 0 \ :\ 1 \leq n \leq N} \quad \text{and}\\
&\qquad  \PP{S_n < -a \ :\ 1 \leq n \leq N} \leq c N^{-(1-H)} \sqrt{\ell(N)}.
\end{aligned}
\end{align}
Thus, one has a lower bound for the probability $\PP{S_n < b \ :\ 1 \leq n \leq N}$, if $b$ is non-negative, and an upper bound, if $b$ is negative. 
In order to get both, a lower estimate and an upper estimate, for some arbitrary $b \in \R$, \cite{Aurzada2016aUnp} uses a change of measure argument. To get this argument to work, a strong assumption on the covariance function of $(S_n)$ is made; namely $\inf_{n \geq 1} \E {S_1 S_n} > 0$ (see also our Remark \ref{rem:fgn} below).
We are able to prove upper and lower bounds whenever \eqref{eq:specDens} is satisfied. We state this result as Corollary \ref{cor:thm11aurzada} below.

The outline of this paper is as follows. In Section 2, we collect some basic properties of the processes $(S_n)$ and $(I_n)$. Moreover, we present some results concerning the reproducing kernel Hilbert spaces of the considered processes that may be of independent interest.
In Section 3, we give a proof of Theorem \ref{thm:S_n}.
Finally, in Section 4, we prove our main result, Theorem \ref{thm:I_n}.

\section{Preliminaries}

Let $(W_H(t))$ be a FBM with Hurst parameter $0<H<1$ and $(I_H(t))$ an IFBM. 
Then, unlike $(W_H(t))$, the process $(I_H(t))$ does not have stationary increments. 
Instead, the process satisfies for all $t_0 \in \R$ \[ \left( I_H(t+t_0) - I_H(t_0)- t W_H(t_0) \right)_{t\in\R} \eqd \left(I_H(t)\right)_{t\in\R} .\] 
In the discrete-time setup, we have analogous properties. 
From the definition of the process $(S_n)$, we straightforwardly obtain stationary increments
\begin{equation*}
\left( S_{n_0+n}-S_{n_0} \right)_{n \in \Z} \eqd \left( S_n \right)_{n \in \Z} \quad \text{for all} \quad n_0 \in \mathbb{Z}.
\end{equation*} 
Also, it is easy to verify that we have 
\begin{equation}
\label{eq:tt}
\left( I_{n_0+n}-I_{n_0}-n\tilde{S}_{n_0} \right)_{n \in \Z} \eqd \left( I_n \right)_{n \in \Z}\quad \text{for all} \quad n_0 \in \mathbb{Z},
\end{equation}
where $(\tilde{S}_n)_{n \in \Z}$ denotes the sequence given by $\tilde{S}_n := \frac{S_n + S_{n-1}}{2}$.

Let us now recall the definition of the reproducing kernel Hilbert space (RKHS) of a centered Gaussian process $(X_t)_{t \in \mathbb{T}}$. For this purpose, let $\mathbb{H}$ denote the $L^2$-closure of the set $\spn\{X_t : t \in \mathbb{T}\}$. Then the RKHS $\mathcal{H}$ of $(X_t)$ is the Hilbert space of functions  \[\mathbb{T} \ni t \mapsto \EE{X_t h}, \quad h \in \mathbb{H},\] with inner product $\left< \EE{Xh_1} , \EE{Xh_2} \right>_{\mathcal{H}} = \EE{h_1 h_2} $. 

The following result, Proposition 1.6 in \cite{Aurzada2013}, will be an important tool
throughout this work.
\begin{prop}
\label{prop:changeOfMeasure}
Let $X$ be some centered Gaussian process 
with RKHS $\mathcal{H}$. Denote by $\|\cdot\|$ the norm in $\mathcal{H}$. Then, for each $f \in \mathcal{H}$ and each measurable $S$ such that
$\PP{X \in S} \in (0, 1)$, we have
\begin{equation}
\label{eq:changeOfMeasure1}
e^{-\sqrt{2 \|f\|^2 \log(1/\PP{X \in S})} - \|f\|^2/2} \PP{X \in S} \leq \PP{X+ f \in S}.
\end{equation}
If $\|f\|^2 < 2 \log(1/\PP{X \in S})$, we have in addition
\begin{equation}
\label{eq:changeOfMeasure2}
\PP{X+ f \in S} \leq e^{\sqrt{2 \|f\|^2 \log(1/\PP{X \in S})} - \|f\|^2/2} \PP{X \in S}.
\end{equation}
\end{prop}

\begin{rem}
We want to mention that the proof of Proposition 1.6 in \cite{Aurzada2013} fails if $\|f\|^2 \geq 2 \log(1/\PP{X \in S})$. 
Thus, unlike in \cite{Aurzada2013}, we have excluded this case here. In the applications of this proposition that we know of, the function $f \in \mathcal{H}$ is fixed and one is interested in the asymptotic behavior of the probabilities $\PP{X \in S^{(N)}}$ for $N\to\infty$, where $(S^{(N)})$ is a sequence of measurable sets such that $\lim_{N\to\infty} \PP{X \in S^{(N)}} = 0$. 
In this case the condition is satisfied for $N$ large enough. Hence, Proposition 1.6 in \cite{Aurzada2013} can be applied in the same way as before.
\end{rem}

First, we show the existence of a function in the RKHS of $(\xi_n)_{n\in\Z}$ with certain asymptotic behavior.

\begin{prop}
\label{prop:RKHS}
Let $H \in (0,1)$, $\rho \in (-1,H-1)$ and let $\mathcal{H}_H(\xi)$ denote the RKHS of the process $(\xi_n)_{n \in \Z}$. Then, if \eqref{eq:specDens} is satisfied, there is an even function $h \in \mathcal{H}_H(\xi)$ such that $h>0$ and $h(n) \sim n^\rho$.
\end{prop}

\begin{proof}
Recall that $h \in \mathcal{H}_H(\xi)$ if and only if there is a function $\varphi \in L^2(\mu)$ with $h(n) = \int_{(-\pi,\pi]} \varphi(u) e^{-inu} \di\mu(u)$, see e.g.\ Comment 2.2.2 (c) in \cite{Ash1975}. %Ash 2.2.2 (c) + Isomorphism (p. 59)
In order to prove the Proposition, we will first consider a function $\varphi_1 \in L^2(\mu)$ such that the corresponding function $h_1 \in \mathcal{H}_H(\xi)$ has the correct asymptotic behavior.
This function can attain non-positive values at finitely many times. To fix this, we will construct afterwards another function $\varphi_2 \in L^2(\mu)$ such that the corresponding function $h_2 \in \mathcal{H}_H(\xi)$ is non-negative, takes positive values when $h_1$ takes non-positive values and decays faster than $h_1$. 
Then, for suitable constants $c_1,c_2>0$, the function $h = c_1 h_1 + c_2 h_2$  has the required properties.

\emph{Construction of $h_1$:}
Due to \eqref{eq:specDens}, there is a function $\tilde{\ell}$ and a constant $u_0>0$ such that $p(u)=\tilde{\ell}(u) \vert u \vert^{1-2H}$ for $u \in [-u_0,u_0]$ and $\tilde{\ell}$ is slowly varying at zero. By Potter's theorem, see 
 Theorem 1.5.6 in \cite{Bingham1987}, $u_0$ can be chosen such that $\tilde{\ell}(u_0)/\tilde{\ell}(u) \leq A \left( \frac{\vert u \vert}{u_0} \right)^{-\delta}$ for $\vert u \vert < u_0$, fixed $A>1$ and fixed $0 < \delta < 2(H-1-\rho)$. 
We set
\begin{equation*}
\varphi_1(u) := \begin{cases} {\vert u \vert^{2H-2-\rho}}/{\tilde{\ell}(u)}, & u \in [-u_0,u_0] \cap \text{supp}( \mu_s)^C,\\ 0, & \text{otherwise}. \end{cases}
\end{equation*}
Then, $\varphi_1 \in L^2(\mu)$ because
\begin{align*}
\int_{(-\pi,\pi]} \vert \varphi_1(u) \vert^2 \di \mu (u)
&= 
\int_{-u_0}^{u_0} \frac{\vert u \vert^{2H -3 - 2\rho  }}{\tilde{\ell}(u)}  \di u \\
&\leq
\frac{A}{\tilde{\ell}(u_0)} \int_{-u_0}^{u_0} \vert u \vert^{2 H -3 -2\rho } \left( \frac{\vert u \vert}{u_0} \right)^{-\delta} \di u < \infty.
\end{align*}
Here we used that $2H-3-2\rho-\delta>-1$.
Moreover,
\begin{align*}
\int_{(-\pi,\pi]} \cos(nu) \varphi_1(u)  \di \mu(u) 
&=
\int_{-u_0}^{u_0} \cos(nu) \vert u \vert^{-\rho - 1 } \di u \\
&=
n^{\rho} \int_{-n u_0}^{n u_0} \cos(v) \vert v \vert^{-\rho-1} \di v \\
&=
2 n^{\rho} \int_{0}^{n u_0} \cos(v) \vert v \vert^{-\rho-1} \di v .
\end{align*}
Since $-\rho-1<0$, it is easy to show, using the Leibniz criterion and the concavity of $(\cdot)^{-\rho-1}$, that the latter integral converges to a constant $c/2>0$. 
Thus, 
\begin{equation*}
h_1(n) = \int_{(-\pi,\pi]} \varphi_1(u) e^{-inu} \di \mu(u) \sim c  n^\rho.
\end{equation*}

\emph{Construction of $h_2$:} 
Choose $n_0$ such that $h_1$ attains only positive values for $\vert n \vert > n_0$.
Let $g \in C^1$ be an even real-valued function with support contained in $[-u_0/2,u_0/2]$ such that the Fourier coefficients for $\vert n \vert \leq n_0$ do not vanish, e.g. take any smooth even function $g$ with $g(u)>0$ for $\vert u \vert < \min(u_0/2,\pi/(2n_0))$ and $g(u)=0$ otherwise.  Then, the function $f $ given by $f(u):=\frac{1}{2\pi}\int_{-\pi}^\pi g(v)\overline{g}(u-v)\di v$ 
has Fourier coefficients $\hat{f}_n = \vert \hat{g}_n \vert^2$. In particular $\hat{f}_n > 0$ for $\vert n \vert \leq n_0$. 
Moreover, $f \in C^2$ because $f$ is a convolution of two differentiable functions. Thus, we have
\begin{equation*}
0 \leq \hat{f}_n = \frac{1}{(in)^2} (\widehat{f''})_n \leq \frac{\sup_{x \in (-\pi,\pi]} \vert f''(x) \vert}{\vert n \vert^2} \quad \text{for} \quad n \in \Z \setminus \{0\}.
\end{equation*}
Now, we consider the function
\begin{align*}
\varphi_2(u) := \begin{cases} \frac{f(u)}{\vert u \vert^{1-2H} \tilde{\ell}(u)}, &  u \in [-u_0,u_0] \cap \text{supp}( \mu_s)^C,\\ 0 ,& \text{otherwise} . \end{cases} 
\end{align*}
Let $M$ denote the maximum of $f$,
then
\begin{align*}
\int_{(-\pi,\pi]} \vert \varphi_2(u) \vert^2 \di \mu(u)
&\leq
\int_{-u_0}^{u_0} \frac{M^2}{\vert u \vert^{1-2H} \tilde{\ell}(u)} \di u \\
&\leq
\frac{A}{\tilde{\ell}(u_0)} \int_{-u_0}^{u_0} \frac{M^2}{\vert u \vert^{1-2H} } \left( \frac{\vert u \vert}{u_0} \right)^{-\delta} \di u
<
\infty,
\end{align*}
since $2H-1-\delta>-1$.
Furthermore, we have by construction of $\varphi_2$
\begin{align*}
h_2(n)
&=
\int_{(-\pi,\pi]} \varphi_2(u) e^{-inu} \di \mu(u)
=
\int_{-\pi}^{\pi} f(u) e^{-inu} \di u = \hat{f}_n. 
\end{align*}
\end{proof}

As a corollary of Proposition \ref{prop:RKHS}, we show the existence of functions with certain asymptotic behavior in the RKHSs of $(S_n)_{n \in \Z}$ and $(I_n)_{n \in \Z}$, respectively.

\begin{cor}
\label{cor:RKHS}
Let $H\in(0,1)$, $\rho \in (-1,H-1)$ and let $\mathcal{H}_H(S)$ and $\mathcal{H}_H(I)$ denote the RKHS of the processes $(S_n)_{n \in \Z}$ and $(I_n)_{n \in \Z}$, respectively. Then, if \eqref{eq:specDens} is satisfied, 
there are functions $f \in \mathcal{H}_H(S)$, $g \in \mathcal{H}_H(I)$ such that $f$ is odd with $f(n) > 0$ for $n>0$ and $f(n) \sim n^{\rho+1}$ as $n \to \infty$ whereas $g$ is even and positive on $\Z \setminus \{0\}$ with $g(n) \sim n^{\rho+2}$ as $n \to \infty$.
\end{cor}

\begin{proof}
Let $h \in \mathcal{H}_H(\xi)$ be the positive and even function in Proposition~\ref{prop:RKHS} with $h(n) \sim n^\rho$. Then, by the definition of the RKHS, there is a random variable $X$ in the $L^2$-closure of the set $\spn\{\xi_n : n \in \Z\}$ 
with $h(n)=\EE{\xi_n X}$.
Now, let the functions $f,g$ be given by $f(n) ={(\rho+1)}\EE{S_n X}$ and $g(n) = {(\rho+1)(\rho+2)}\EE{I_n X}$, respectively.
Since the sets $\spn\{\xi_n : n \in \Z\}$, $\spn\{S_n : n \in \Z\}$ and $\spn\{I_n : n \in \Z\}$ coincide, we have $f \in \mathcal{H}_H(S)$ and $g \in \mathcal{H}_H(I)$.  
By $h(n) \sim n^\rho$ and the symmetry of $h$, we have $-f(-n) = f(n) = (\rho+1)\sum_{k=1}^n h(k) \sim n^{\rho+1}$ as $n \to \infty$. Thus, we have further $g(-n) = g(n) = (\rho+1)(\rho+2) \sum_{k=1}^{n-1} \EE{S_k X} + \EE{S_n X}/2 \sim n^{\rho+2}$ as $n \to \infty$.
\end{proof}

As a first application of Corollary \ref{cor:RKHS}, we compare the persistence probabilities of $(I_n)$ to a closely related process.
Let $(\bar{I}_n)_{n\in\Z}$ be the sequence given by $\bar{I}_n-\bar{I}_{n-1}:=S_n$ for $n \in \Z$ and $\bar{I}_0:=0$. This process is related to the process $(I_n)$ by the identity $\bar{I}_n = I_n + S_n/2$. Both processes are defined as integrals of stationary increments sequences that have FBM as scaling limit. In the context of this paper, the major difference between these processes is that $(I_n)$ vanishes only at $0$ and  satisfies $(I_n) \eqd (I_{-n})$ whereas $\bar{I}_{-1}=\bar{I}_0=0$ and $\bar{I}_1$ does not vanish. The symmetry property of $(I_n)$ resembles the continuous-time case and is needed in the proof of Theorem \ref{thm:I_n}. In the following corollary, we relate the persistence probabilities of both processes.

\begin{cor}
Let $(\xi_n)$ be a real valued stationary centered Gaussian sequence such that (\ref{eq:specDens}) holds. Then,
\[\PP{\bar{I}_n \leq 0 \ :\  -N-1 \leq n  \leq N} \leq \PP{I_n \leq 0 \ :\  \vert n \vert \leq N}.\]
If in addition $\EE{ \bar{I}_n \bar{I}_m} \geq 0$ for all $n,m \in \Z$, then one has
\begin{align*}
\PP{I_n \leq 0 \ :\  \vert n \vert \leq N}
\leq 
\PP{\bar{I}_n \leq 0 \ :\  \vert n \vert \leq N} \ell_0(N),
\end{align*}
where $\ell_0$ denotes a slowly varying function at infinity.
\end{cor}

\begin{proof}
The first inequality follows directly from the definitions of the processes, since one has $I_n=(\bar{I}_n+\bar{I}_{n-1})/2$ for all $n \in \Z$. Using Slepian's Lemma and the additional assumption about the correlations of $(\bar{I}_n)$, we obtain
\begin{align}
\begin{aligned}
\label{eq:IbarSlepian}
\PP{\bar{I}_n \leq 0 \ :\ \vert n \vert \leq N} 
\geq & \,
\PP{\bar{I}_n \leq 0 \ :\ \vert n \vert \leq \log(N)}\\
&\cdot \PP{\bar{I}_n \leq 0 \ :\ \log(N) < \vert n \vert \leq N}.
\end{aligned}
\end{align}
By the same argument and Theorem \ref{thm:S_n}, we have
\begin{align*}
\PP{\bar{I}_n \leq 0 \ :\ \vert n \vert \leq \log(N)} 
\geq & \,
\PP{\bar{I}_n \leq 0 \ :\ 0 \leq n  \leq \log(N)}\\
&\cdot \PP{\bar{I}_n \leq 0 \ :\  -\log(N) \leq n < 0 } \\
\geq & \,
\PP{S_n \leq 0 \ :\ 0 \leq n  \leq \log(N)}\\
&\cdot \PP{S_n \geq 0 \ :\  -\log(N) \leq n < 0 } \\
\geq & \,
c^{-2} \log(N)^{-2}.
\end{align*}
Thus, the first factor on the right hand side in \eqref{eq:IbarSlepian} can be estimated by a slowly varying function at infinity.
It remains to relate the second factor on the right hand side in \eqref{eq:IbarSlepian} to the probability $\PP{I_n \leq 0 \ :\  \vert n \vert \leq N}$.

By Corollary \ref{cor:RKHS}, for $\varepsilon \in (0,1/4)$, there is a symmetric function $f \in \mathcal{H}_H(I)$ such that $f(n) \geq \vert n \vert^{1+H-\varepsilon}$ for all $n \in \Z$. 
Obviously, we have
\begin{align}
\label{eq:IbarDecomp}
\begin{aligned}
&\PP{I_n \leq -n^{1+H-\varepsilon} \ :\ \log(N) < \vert n \vert \leq N} \\
& \qquad \qquad \qquad \leq 
\PP{\bar{I}_n \leq 0 \ :\ \log(N) < \vert n \vert \leq N}\\
&\qquad \qquad \qquad \quad +
\PP{ \exists n \ :\ \bar{I}_n - I_n > n^{1+H-\varepsilon}  ,\ \log(N) < \vert n \vert \leq N}.
\end{aligned}
\end{align}
We will see that the second term on the right hand side is of lower order, while the term on the left hand side can be related to $\PP{I_n \leq 0 \ :\  \vert n \vert \leq N}$. 
For this purpose, let $X$ denote a standard normal random variable. Then, by using $\bar{I}_n-I_n=S_n/2$ in the first step and \eqref{eq:varS} in the second step, we have for $N$ large enough
\begin{align}
\label{eq:IbarHighJumps}
\begin{aligned}
&\PP{ \exists n \ :\ \bar{I}_n - I_n > n^{1+H-\varepsilon}  ,\ \log(N) < \vert n \vert \leq N}\\
& \quad \quad  \leq 
2\sum_{n=\lceil \log(N) \rceil}^N \PP{S_n/2 > n^{1+H-\varepsilon}}\\
& \quad \quad  \leq 
2\sum_{n=\lceil \log(N) \rceil}^N \PP{n^{H+\varepsilon} X > n^{1+H-\varepsilon}}\\
& \quad \quad  \leq 
2N \PP{ X > \log(N)^{1-2\varepsilon}}\\
& \quad \quad  \leq
2N e^{-(\log(N))^{2-4\varepsilon}/2}\\
& \quad \quad  \leq
2N^{-2}.
\end{aligned}
\end{align}
In the fourth step above, we used the standard estimate $\PP{X > x} \leq e^{-x^2/2}$ for $x\geq 1$. 
Finally, using Proposition \ref{prop:changeOfMeasure}, we obtain for $N$ large enough 
\begin{align*}
\PP{I_n \leq 0 \ :\  \vert n \vert \leq N} 
\leq & 
\PP{I_n \leq 0 \ :\ \log(N) < \vert n \vert \leq N}  \\
\leq &
\PP{I_n \leq -f(n) \ :\ \log(N) < \vert n \vert \leq N} \\ & \cdot e^{\sqrt{2\|f\|^2 \log(1/\PP{I_n \leq 0  \ :\ \vert n \vert \leq N})}-\|f\|^2/2} \\
\leq &
\PP{I_n \leq -n^{1+H-\varepsilon} \ :\ \log(N) < \vert n \vert \leq N} \\ & \cdot e^{\sqrt{2\|f\|^2 \log(1/\PP{I_n \leq 0 \ :\ \vert n \vert \leq N})}-\|f\|^2/2}.
\end{align*}
This, together with \eqref{eq:IbarDecomp}, \eqref{eq:IbarHighJumps} and Theorem \ref{thm:I_n}, finishes the proof.
\end{proof}

As another application of Corollary \ref{cor:RKHS}, we can give an improvement of Theorem 11 in \cite{Aurzada2016aUnp}:

\begin{cor}
\label{cor:thm11aurzada}
Let $(\xi_n)$ be a real valued stationary centered Gaussian sequence such that \eqref{eq:specDens} holds. Then, for every $b \in \R$ there is some constant $c>0$ such that
\begin{align*}
N^{-(1-H)} \sqrt{\ell(N)} e^{-c\sqrt{\log(N)}}
&\leq
\PP{\max_{1\leq n \leq N} S_n \leq b}\\
&\leq
N^{-(1-H)}\sqrt{\ell(N)}e^{c \sqrt{\log(N)}} \quad \forall N \in \mathbb{N}.
\end{align*}
\end{cor}

\begin{proof}
Let $a>0$. By Corollary \ref{cor:RKHS}, there is a function $f \in \mathcal{H}_H(S)$ with $f(n) \geq 2a$ for all $n \geq 1$. Further, using the lower estimate in \eqref{eq:thm11aurzada}, we have for $N$ large enough \[  N^{-1} \leq \PP{ S_n \leq a \ :\ 1 \leq n \leq N } .\]
This together with Proposition \ref{prop:changeOfMeasure} yields for $N$ large enough
\begin{align*}
&\PP{S_n \leq -a \ :\ 1 \leq n \leq N}\\
&\qquad\qquad=
\PP{ S_n + f(n) \leq -a + f(n) \ :\ 1 \leq n \leq N }\\
&\qquad\qquad\geq
\PP{ S_n + f(n) \leq a \ :\ 1 \leq n \leq N }\\
&\qquad\qquad\geq
\PP{ S_n \leq a \ :\ 1 \leq n \leq N } e^{-\sqrt{2 \|f\|^2 \log(N) } - \|f\|/2}
.
\end{align*}
Combining this with \eqref{eq:thm11aurzada} finishes the proof.
\end{proof}

\begin{rem}
\label{rem:fgn}
In Theorem 11 in \cite{Aurzada2016aUnp}, the authors assume $\inf_{n \geq 1} \E S_nS_1>0$ to get the change of measure argument to work. For instance, the fractional Gaussian noise process $(\xi_n^{\textsc{fgn}})$ satisfies this assumption. This can be easily verified by using that $\E S_n^2 = n^{2H}$.
In general, this does not remain true if one only has \eqref{eq:varS}. For example, consider the case where $\ell(x) = 1 + {\cos(\pi x )}/{\log(x)}$ in \eqref{eq:varS}. Then, one has $\sum_{j=1}^n \sum_{k=1}^n \E \xi_j \xi_k \sim n^{2H}$ but the function $\E S_n S_1$ attains infinitely often positive and negative values.
\end{rem}

\begin{rem}
Consider the function $f \colon \mathbb{N} \to \R$ with $f(n)=\mathbbm{1}_{n=1}$. 
Clearly, $f$ is in the RKHS of the process $(\xi_n)_{n \geq 1}$ if and only if $\xi_1 \not\in \mathbb{H}_2$, where $\mathbb{H}_2$ denotes the $L^2$-closure of the set $\spn\{\xi_n : n \geq 2\}$.  
It is well known that this condition is equivalent to the Kolmogorov condition
\begin{equation}
\label{eq:kolmogorovsFormula}
\int_{-\pi}^\pi \log(p(u)) \di u > - \infty,
\end{equation}
where $p$ denotes the density of the component of the spectral measure of $(\xi_n)$ that is absolutely continuous with respect to the Lebesgue measure, see e.g. Theorem 2.5.4 in \cite{Ash1975}.
In this case, all constant functions are in the RKHS of the process $(S_n)_{n \geq 1}$.
Hence, the proof of Corollary \ref{cor:thm11aurzada} still works if we replace condition \eqref{eq:specDens}  by \eqref{eq:varS} and \eqref{eq:kolmogorovsFormula}.
\end{rem}

\section{\texorpdfstring{Proof of Theorem~\ref{thm:S_n}}{Proof of Theorem 1}}

\subsection*{Upper bound}
Let $T_N$ denote the time 
where the process $(S_n)_{n\in\mathbb{Z}}$ attains its maximum on $\{0,1,\ldots,N\}$.
Since $(S_n)_{n\in\mathbb{Z}}$ has stationary increments and $\PP{S_j = S_k}=0$ for $j \neq k$, the upper bound follows from 
\begin{align*}
N \cdot \PP{S_n \leq 0 \ :\  -N \leq n \leq N}
&\leq
\sum_{k=1}^N \PP{S_n \leq 0 \ :\  -k \leq n \leq N-k} \\
&=
\sum_{k=1}^N \PP{S_n \leq S_k \ :\  0 \leq n \leq N}\\
&=
\sum_{k=1}^N \PP{T_N=k} \\
&\leq
1.
\end{align*}

\subsection*{Lower bound}
Using again the stationary increments of $(S_n)_{n\in\mathbb{Z}}$, we obtain
\begin{align}
\label{eq:proofSlow}
\begin{aligned}
(N+1)& \cdot \PP{S_n \leq 0 \ :\  -N \leq n \leq N}\\
&\geq
\sum_{k=0}^N \PP{S_n \leq 0 \ :\  -N-k \leq n \leq 2N-k} \\
&=
\sum_{k=0}^N \PP{S_n \leq S_{N+k} \ :\  0 \leq n \leq 3N}\\
&=
\sum_{k=0}^N \PP{T_{3N} = N+k}\\
&=
\PP{T_{3N} \in [N,2N]}.
\end{aligned}
\end{align}
Now, we consider the continuous functional $F \colon ( D([0,1]) , \|\cdot\|_\infty) \to (\mathbb{R},\vert \cdot \vert)$ given by  \[F(g) = \left( \sup_{x \in \left(\frac{1}{3},\frac{2}{3}\right)} g(x) - \sup_{x \in \left(0,\frac{1}{3}\right) \cup \left(\frac{2}{3},1\right)} g(x) \right)_+ \wedge 1, \] 
where 
$(x)_+ := \max(x,0)$ for $x \in \mathbb{R}$ and $D([0,1])$ denotes the set of all c\`adl\`ag functions on $[0,1]$.
We set \[Y_N(t) = \frac{1}{N^H \ell(N)^{1/2}} \sum_{k=1}^{\lfloor Nt \rfloor} \xi_k.\]
Due to (\ref{scalingLimit}), it follows that
\begin{align*}
\PP{T_{3N} \in [N,2N]} &= \EE{\mathbbm{1}_{T_{3N} \in [N,2N]}} \geq \E F\left( Y_N \right) \to c_0 > 0, \quad \text{as } N \to \infty.
\end{align*}
This and \eqref{eq:proofSlow} show the lower bound.

\section{\texorpdfstring{Proof of Theorem~\ref{thm:I_n}}{Proof of Theorem 2}}

The proof is structured as follows: We first consider the functional 
\begin{equation*}
F_N := \sum_{k=1}^{N-1} \left( \gamma_{k,k}^- - \gamma_{k,N-k}^+ \right)_+,
\end{equation*}
where for $k \in \mathbb{Z}$ and $m \in \mathbb{N}$
\begin{align*}
\gamma_{k,m}^- := \min_{1 \leq n \leq m} \frac{I_k-I_{k-n}}{n} 
\quad \text{and} \quad
\gamma_{k,m}^+ := \max_{1 \leq n \leq m} \frac{I_{k+n}-I_{k}}{n},
\end{align*}
and determine the 
polynomial order of $\E F_N$ as $N \to \infty$.
Then, we relate the quantity $\E F_N$ to the probability
\begin{equation}
\label{eq:probProof}
\tilde{p}_N := \PP{I_n + \vert n \vert \leq 0, \vert n \vert \leq N}.
\end{equation}
Finally, we obtain the asymptotic order of 
\begin{equation}
\label{eq:prob}
p_N:=\PP{I_n \leq 0 \ :\ \vert n \vert \leq N}
\end{equation}
from (\ref{eq:probProof}) by using a change of measure argument (Proposition \ref{prop:changeOfMeasure} and Corollary \ref{cor:RKHS}).

\subsection*{Upper bound for $\E F_N$}
In the following, we fix $N$ and write $\gamma_k^-=\gamma_{k,k}^-$ and $\gamma_k^+=\gamma_{k,N-k}^+$ to ease notation.
Let $C_N \colon [0,N] \to \R$ denote the concave majorant of $I_n$ on $[0,N]$, i.e., $C_N$ is the smallest concave function with $I_n \leq C_N(n)$. Obviously, $C_N$ is a piecewise linear function and we denote by $\{k_{1},k_{2},\ldots\}$ (depending on $N$) its nodal points. At these points the slope on the left is $\gamma_{k_{i}}^-$ and the slope on the right is $\gamma_{k_{i}}^+$. Further, we note that $\gamma_{k}^- - \gamma_{k}^+ \geq 0$ if and only if $k$ is a nodal point of $C_N$. In that case one has $\gamma_{k_{i}}^+=\gamma_{k_{i+1}}^-$. Thus,
\begin{equation*}
F_N =  \sum_{k=1}^{N-1} \left( \gamma_{k}^- - \gamma_{k}^+ \right)_+ = \sum_i \left( \gamma_{k_{i}}^- - \gamma_{k_{i+1}}^- \right) = \gamma_{0}^+ - \gamma_{N}^-.
\end{equation*}
By $\E \tilde{S}_N = 0$, (\ref{eq:tt}) and $(I_n) \eqd (I_{-n})$, we have
\begin{align*}
\EE{ -\gamma_{N}^-}
&= \EE{-\min_{1 \leq n \leq N} \frac{I_N-I_{N-n}}{n} }  \\
&= \EE{ \max_{1 \leq n \leq N} \frac{I_{N-n}-I_N-(-n)\tilde{S}_N }{n}} \\
&= \EE{ \max_{1 \leq n \leq N} \frac{I_{-n}}{n}} = \EE{ \max_{1 \leq n \leq N} \frac{I_n}{n}} = \E \gamma_{0}^+.
\end{align*}
Therefore,
\begin{equation}
\label{eq:EF}
\E F_N = 2 \E \gamma_{0}^+.
\end{equation}
Due to (\ref{eq:EF}), one obtains the upper estimate
\begin{align*}
\E F_N &= 2 \EE{ \max_{1 \leq n \leq N} \frac{\sum_{k=1}^n \tilde{S}_k}{n}} \\
&\leq 2 \EE{ \max_{1 \leq n \leq N} \frac{\sum_{k=1}^n \max_{1 \leq j \leq N }\tilde{S}_j}{n}}
= 2 \EE{ \max_{1 \leq j \leq N} \tilde{S}_j}.
\end{align*}
It can be obtained from (\ref{scalingLimit}) that
\begin{equation*}
\frac{1}{N^H \ell(N)^{1/2}} \EE{ \max_{1 \leq n \leq N} \tilde{S}_n} \to \EE{ \sup_{t \in [0,1] } W_H(t)} \in (0,\infty),
\end{equation*}
where $(W_H(t))$ is a fractional Brownian motion, see e.g.\ proof of Theorem~11 in \cite{Aurzada2016aUnp}. 
Thus, there is a constant $c$ such that for all $N$
\begin{equation}
\label{eq:EFleq}
\E F_N \leq c \, \ell(N)^{1/2} N^H . 
\end{equation}
In the following, $c$ will denote a varying positive constant independent of $N$ for ease of notation.

\subsection*{Lower bound for $\E F_N$}

Since $(\xi_n)$ is a stationary process, we have 
\[ \E S_j S_k  = \frac{1}{2} \left( \E S_j^2 + \E S_k^2 - \E S_{\vert j-k \vert}^2 \right) .\]
Consequently,
\[
\E \left(I_N+\frac{S_N}{2}\right)^2 = \sum_{j=1}^N \sum_{k=1}^N \E S_j S_k = \frac{1}{2} \sum_{j=1}^N \sum_{k=1}^N \left( \E S_j^2 + \E S_k^2 - \E S_{\vert j-k \vert}^2 \right). 
\]
Counting how often $\E S_k^2$ is added, yields  
\begin{align*}
\E \left(I_N+\frac{S_N}{2}\right)^2 &= 
\frac{1}{2} \sum_{k=1}^N N \E S_k^2 + \frac{1}{2} \sum_{k=1}^N N \E S_k^2 - \sum_{k=1}^N (N-k) \E S_k^2\\
&= \sum_{k=1}^N k \E S_k^2.
\end{align*}
Since $\E S_n^2 \sim n^{2H}\ell(n)$, we can apply Proposition 1.5.8 in \cite{Bingham1987} to obtain
\begin{equation}
\label{eq:VarInApriori}
\E \left(I_N+\frac{S_N}{2}\right)^2 \sim N^{2H+2} \ell(N) / (2H+2).
\end{equation}
Now, using the Cauchy-Schwarz Inequality, we have $\left\vert \E S_N I_N \right\vert \leq \sqrt{\E S_N^2 \E I_N^2}$ and we can thus conclude from \eqref{eq:VarInApriori} that 
\begin{equation}
\label{eq:EI_n}
\E I_N^2 \sim N^{2H+2} \ell(N) / (2H+2).
\end{equation}
In the following, we let $\| \cdot \|_2$ denote the norm $\|X\|_2=\EE{\vert X \vert^2}^{1/2}$. Moreover, we recall the identity $\E X_+ = (2 \pi)^{-1/2} \| X \|_2$ for a centered normal random variable $X$.
Now, we can give a lower bound for $\E F_N$. By (\ref{eq:EF}) and $\E I_1 = 0$, we have
\begin{align*}
\E F_N &= 2 \EE{\max_{1 \leq n \leq N} \frac{I_n}{n}}
=
2 \EE{\max_{1 \leq n \leq N} \frac{I_n}{n}-I_1} \\
&=
2 \E\left({\max_{1 \leq n \leq N} \frac{I_n}{n}-I_1}\right)_+ 
\geq
2 \E\left({\frac{I_N}{N}-I_1}\right)_+ \\
&=
\sqrt{2 / \pi} \left\| \frac{I_N}{N} - I_1 \right\|_2 
\geq
\sqrt{2/\pi} \left( \left\| \frac{I_N}{N} \right\|_2 - \left\|  I_1 \right\|_2 \right).
\end{align*}
Thus, by (\ref{eq:EI_n}), we have
\begin{equation}
\label{eq:EFgeq}
\E F_N \geq c^{-1} \ell(N)^{1/2} N^{H}.
\end{equation}

\subsection*{Upper bound for $\tilde{p}_N$}

In order to get an upper bound for the probability in \eqref{eq:probProof}, it is convenient to consider the random variable 
\begin{equation*}
\vartheta_N := \left( \gamma_{0,N}^- - \gamma_{0,N}^+ \right)_+.
\end{equation*}
We have
\begin{equation}
\label{eq:ineqVartheta}
\E \left( \gamma_{k}^- - \gamma_{k}^+ \right)_+ \geq \E{\vartheta_N}.
\end{equation}
To see this, note that by using (\ref{eq:tt}), we obtain
\begin{align*}
\gamma_{k,N}^- - \gamma_{k,N}^+ &= \min_{1 \leq n \leq N} - \frac{I_{k-n}-I_k-(-n)\tilde{S}_k}{n} - \max_{1 \leq n \leq N} \frac{I_{k+n}-I_k-n\tilde{S}_k}{n} \\
&\eqd \min_{1 \leq n \leq N} - \frac{I_{-n}}{n} - \max_{1 \leq n \leq N} \frac{I_n}{n}
= \gamma_{0,N}^- - \gamma_{0,N}^+.
\end{align*}
Combining this with $\gamma_k^- - \gamma_k^+ \geq \gamma_{k,N}^- - \gamma_{k,N}^+ $ shows (\ref{eq:ineqVartheta}).
Applying the Markov Inequality, we see that
\begin{equation}
\label{eq:chebyshev}
\E{\vartheta_N} \geq 2 \PP{ \vartheta_N \geq 2 }.
\end{equation}
Using \eqref{eq:EFleq}, \eqref{eq:ineqVartheta}, and \eqref{eq:chebyshev}, we thus obtain
\begin{align}
\label{eq:upperBound}
\begin{aligned}
c \, \ell(N)^{1/2} N^H  &\geq \E {F_N} = \sum_{k=1}^{N-1} \E \left( \gamma_k^- - \gamma_k^+ \right)_+ \\
&\geq (N-1) \E{\vartheta_N} \geq 2(N-1) \PP{ \vartheta_N \geq 2 } \\
&\geq  2(N-1) \PP{ \gamma_{0,N}^- \geq 1, \gamma_{0,N}^+ \leq -1 } \\
&=  2(N-1) \PP{ I_n + \vert n \vert \leq 0 \ :\ \vert n \vert \leq N }.
\end{aligned}
\end{align}
Hence, we have for any $N$ 
\begin{equation}
\label{eq:pleq}
\PP{ I_n + \vert n \vert \leq 0 \ :\ \vert n \vert \leq N }
\leq c \, \ell(N)^{1/2} N^{-(1-H)} .
\end{equation}

\subsection*{Lower bound for $\tilde{p}_N$}
Along the lines of the proof of \eqref{eq:ineqVartheta}, one gets an analogous estimate when replacing $N$ by $\tilde{k} := \min(k,N-k)$; namely
\begin{equation}
\label{eq:ineqVartheta2}
\E \left( {\gamma_{k}^- - \gamma_{k}^+} \right)_+ \leq \E{\vartheta_{\tilde{k}}}.
\end{equation}
Now, let $1-H < \alpha < 1$. Then, we have by the monotonicity of $\vartheta_N$ for $N^\alpha \leq k \leq N-N^\alpha$
\begin{equation}
\label{eq:ineqVartheta2.5}
\E{\vartheta_{\tilde{k}}} \leq \E{ \vartheta_{\left\lceil N^\alpha \right\rceil } }.
\end{equation}
Thus, by using \eqref{eq:ineqVartheta2} and \eqref{eq:ineqVartheta2.5}, we obtain 
\begin{align*}
\begin{aligned}
\E F_N
&=
\sum_{k=1}^{N-1} \E \left({ \gamma_k^- - \gamma_k^+ }\right)_+ \\
&\leq
\left(N-2 \left\lfloor N^\alpha \right\rfloor \right) \E{ \vartheta_{\left\lceil N^\alpha \right\rceil } }
+ 2\sum_{k=1}^{\left\lfloor N^\alpha \right\rfloor} \E{ \vartheta_{k} }. 
\end{aligned}
\end{align*}
Moreover, we know from (\ref{eq:upperBound}), that we have for all $k$ 
\begin{equation}
\label{eq:ineqVartheta3}
\E \vartheta_k \leq c \, \ell(k)^{1/2} k^{H-1}.
\end{equation}
Hence, by (\ref{eq:ineqVartheta3}) and Proposition 1.5.8 in \cite{Bingham1987}, we obtain 
\[
\sum_{k=1}^{\left\lfloor N^\alpha \right\rfloor} \E{ \vartheta_{k} } \leq c \, \ell({\left\lfloor N^\alpha \right\rfloor})^{1/2}  N^{\alpha H}.
\]
Thus, we have
\[
c^{-1} \ell(N)^{1/2} N^H \leq \E F_N \leq N \E{ \vartheta_{\left\lceil N^\alpha \right\rceil} } + c\, \ell({\left\lfloor N^\alpha \right\rfloor})^{1/2} N^{\alpha H}.
\]
Since $\alpha H < 1$, we obtain 
\[c^{-1}  \ell(N)^{1/2} N^{H-1} \leq \E{ \vartheta_{\left\lceil N^\alpha \right\rceil} }.\]
Replacing $N$ by $\left\lceil N^{1/\alpha} \right\rceil$ yields 
\[  \ell_1(N) N^{-(1-H)/\alpha} \leq \E{ \vartheta_{N} },\]
where $\ell_1$ is a slowly varying function at infinity. 

Fix $q>1$ to be chosen later and let $\|\cdot\|_q$ denote the norm $\EE{\vert X \vert^q}^{1/q}$ for some random variable $X$.
Then, using $\vartheta_N \leq \vartheta_1$ and H\"older's Inequality, we have \[\E \vartheta_N = \E \vartheta_N \mathbbm{1}_{\vartheta_N > 0}\leq \E \vartheta_1 \mathbbm{1}_{\vartheta_N > 0} \leq \| \vartheta_1 \|_q \PP{\vartheta_N > 0}^{1-1/q}.\]
Further, 
\[\| \vartheta_1 \|_q \leq \| I_{-1} - I_{1} \|_q \leq c \sqrt{q}, 
\] using that $I_{-1} - I_{1}$ is a Gaussian random variable.
So, we have
\begin{equation*}
\frac{\ell_1(N)}{c \sqrt{q}} N^{-(1-H)/\alpha} \leq \PP{\vartheta_N > 0}^{1-1/q}.
\end{equation*}
Now, setting $q:=\log(N)+1$ yields 
\begin{equation}
\label{eq:varthAlpha}
\ell_2(N)N^{-(1-H)/\alpha} \leq \PP{\vartheta_N > 0},
\end{equation}
where $\ell_2$ is a slowly varying function at infinity.

In the following, we will relate the probability $\PP{\vartheta_N>0}$ to the probability in (\ref{eq:probProof}).
This is divided into four steps.

\emph{Step 1:} %A change of measure argument.
We start with a change of measure argument.
By Corollary \ref{cor:RKHS}, we can find a function  $f \in \mathcal{H}_H(I)$ such that $f(n) \geq \frac{3}{2}\vert n \vert$ for all $n \in \Z$.
Then, using (\ref{eq:changeOfMeasure1}),
we obtain
\begin{align}
\begin{aligned}
\label{eq:com}
&e^{-\sqrt{2\|f\|^2  \log(1/\PP{\vartheta_N > 0})}-\|f\|^2/2} \PP{\vartheta_N > 0}
\\& \qquad \qquad \qquad \leq
\PP{\min_{-N \leq n \leq -1} \frac{I_n + f(n)}{n} - \max_{1 \leq n \leq N} \frac{I_n + f(n)}{n} > 0} \\
& \qquad \qquad \qquad \leq
\PP{\min_{-N \leq n \leq -1} \frac{I_n + \frac{3}{2}\vert n \vert}{n} - \max_{1 \leq n \leq N} \frac{I_n + \frac{3}{2}\vert n \vert}{n} > 0}\\
& \qquad \qquad \qquad =
\PP{\min_{-N \leq n \leq -1} \frac{I_n}{n} - \max_{1 \leq n \leq N} \frac{I_n}{n} > 3}\\
& \qquad \qquad \qquad =
\PP{ \vartheta_N > 3 }.
\end{aligned}
\end{align}
So, by (\ref{eq:varthAlpha}), $ \PP{\vartheta_N > 0}$ and $\PP{ \vartheta_N > 3 }$ differ by less than a slowly varying function at infinity.

\emph{Step 2:}
Let
\[
A_0^{(N)} := \left\{ (x_{-N},\ldots,x_{-1},x_1,\ldots,x_N) \in \R^{2N} \ :\ x_n \leq -\vert n \vert , 1 \leq \vert n \vert \leq N \right\},
\]
and let
\[
A_m^{(N)} := A_0^{(N)} + m b^{(N)}, \quad \text{where } b^{(N)}:=(-N,\ldots,-1,1,\ldots,N).
\]
In the following, we write $I \in A_m^{(N)}$ instead of $(I_{-N},\ldots,I_{-1},I_1,\ldots,I_N) \in A_m^{(N)}$ for ease of notation. We will show that $\{\vartheta_N > 3\} \subseteq \cup_{m \in \Z} \{ I \in A_m^{(N)} \}$. For this purpose, let $m^{(N)}$ be an integer-valued random variable such that  $\min_{-N \leq n \leq -1} \frac{I_n}{n}  \in [m^{(N)}+1,m^{(N)}+2)$. Then, we obviously have $I_n \leq (m^{(N)}+1)n$ for $-N \leq n \leq -1$. Furthermore, assuming $\vartheta_N > 3$, we can conclude that \[\max_{1 \leq n \leq N} \frac{I_n}{n} < \min_{-N \leq n \leq -1} \frac{I_n}{n} - 3 < m^{(N)} -1 .\] 

\emph{Step 3:}
We show that $\PP{I \in A_0^{(N)}} \geq \PP{I \in A_m^{(N)}}$. For this purpose, we make use of an argument that is commonly used to prove Anderson's Inequality.
It is well known that for any convex subsets $A,B \subseteq \R^{2N}$ and $0 < \lambda < 1$, one has \[
\mu \left( \lambda A + (1-\lambda) B \right) 
\geq 
\mu \left(  A  \right)^\lambda
\mu \left(  B  \right)^{1-\lambda},
\]
where $\mu$ is a centered Gaussian measure on $\R^{2N}$, see e.g. Theorem 2 in \cite{Prekopa1971}.
Since $(I_{-N},\ldots,I_{-1},I_1,\ldots,I_N)$ is a centered Gaussian random variable, by setting $\lambda = \frac{1}{2}$, we obtain 
\begin{align*}
\PP{I \in A_0^{(N)}} &= \PP{ I \in \frac{1}{2} A_{-m}^{(N)} + \frac{1}{2} A_{m}^{(N)} } \\ &\geq \PP{ I \in A_{-m}^{(N)}}^{1/2} \PP{I \in A_{m}^{(N)}}^{1/2}=\PP{I \in A_m^{(N)}}.
\end{align*} 
Here, we used that one has $A_0^{(N)} = \frac{1}{2} A_{-m}^{(N)} + \frac{1}{2} A_{m}^{(N)}$ and, by symmetry of the process $(I_n)$, $\PP{ I \in A_{-m}^{(N)}} = \PP{I \in A_{m}^{(N)}}$.

\emph{Step 4:}
Now, we relate the quantities $\PP{\vartheta_N > 3}$ and $\tilde{p}_N$.
Since $I_{-1}$ is a centered Gaussian random variable, we can choose a constant $c_0$ such that $\PP{I_{-1} \leq -(a_N + 1)} \in o(N^{-1})$ for $a_N = \sqrt{c_0 \log(N) }$. Further, by $\PP{\cup_{m \geq a_N} A_m} \leq \PP{I_{-1}<-(a_N + 1)}$ and symmetry of the process $(I_n)$, we get
\begin{equation*}
\PP{\cup_{\vert m \vert \geq a_N} A_m} \in o(N^{-1}).
\end{equation*}
Altogether we thus obtain
\begin{align*}
\begin{aligned}
\PP{\vartheta_N > 3}
&\leq
\PP{ \cup_{m \in \Z} A_m } \\
&\leq
\sum_{\vert m \vert < a_N} \PP{A_m} + \PP{\cup_{\vert m \vert \geq  a_N} A_m } \\
&\leq
2a_N \PP{A_0} + 2\PP{ I_{-1} \leq -(a_N+1) } \\
&=
2a_N \PP{I_n + \vert n \vert \leq 0 \ :\ \vert n \vert \leq N} + o(N^{-1}).
\end{aligned}
\end{align*}
Putting this together with \eqref{eq:varthAlpha} and \eqref{eq:com}, we get
\begin{equation}
\label{eq:pgeq}
\ell_3(N) N^{-(1-H)/\alpha} \leq 
\PP{I_n + \vert n \vert \leq 0 \ :\ \vert n \vert \leq N},
\end{equation}
where $\ell_3$ denotes a slowly varying function at infinity.

\subsection*{Polynomial rate of $p_N$}

Clearly, we have from \eqref{eq:pgeq}
\begin{align}
\label{eq:probLowerBound}
\begin{aligned}
\ell_3(N) N^{-(1-H)/\alpha}
&\leq
\PP{I_n + \vert n \vert \leq 0 \ :\ \vert n \vert \leq N} \\
&\leq
\PP{I_n \leq 0 \ :\ \vert n \vert \leq N} = p_N.
\end{aligned}
\end{align}
In particular, $p_N \geq c^{-1}N^{-1}$ for some suitable constant $c$. This estimate will be used in the following change of measure argument.
Due to Corollary \ref{cor:RKHS}, we can choose a function $f \in \mathcal{H}_H(I)$ with $f(n) \geq \vert n \vert$ for all $n \in \Z$. Then, by (\ref{eq:pleq}) 
and Proposition \ref{prop:changeOfMeasure}, we obtain
\begin{align}
\label{eq:probUpperBound}
\begin{aligned}
c \, \ell(N)^{1/2}  N^{-(1-H)}
&\geq
\PP{I_n + \vert n \vert \leq 0 \ :\ \vert n \vert \leq N} \\
&\geq
\PP{I_n + f(n) \leq 0 \ :\  \vert n \vert \leq N} \\
&\geq
\PP{I_n \leq 0 \ :\  \vert n \vert\leq N} e^{-\sqrt{2\|f\|^2 \log(1/p_N)}-\|f\|^2/2}
\\
&\geq
\PP{I_n \leq 0 \ :\  \vert n \vert \leq N} e^{-\sqrt{2\|f\|^2 \log(cN)}-\|f\|^2/2} .
\end{aligned}
\end{align}
Finally, we take $\log$ in \eqref{eq:probLowerBound}, \eqref{eq:probUpperBound}  and divide by $\log(N)$. Then, taking $\limsup_N$ and $\liminf_N$, respectively, and letting $\alpha \nearrow 1$ yields
\begin{align*}
\lim_{N \to \infty} \frac{\log\left( \PP{I_n \leq 0 \ :\ \vert n \vert \leq N} \right)}{\log(N)} = H-1.
\end{align*}

\end{document}